\documentclass[reqno, 11pt]{amsart}

\usepackage{amssymb, amsmath, amsthm, mathtools, floatflt, palatino, euler, epic, eepic, microtype, accents}
\usepackage{cases}
\usepackage{graphicx}
\usepackage{float}
\usepackage[usenames]{xcolor}

\usepackage{subcaption}

\makeatletter
\def\BState{\State\hskip-\ALG@thistlm}
\makeatother

\makeatletter
\def\BState{\State\hskip-\ALG@thistlm}
\makeatother

\usepackage[margin=1 in]{geometry}
\definecolor{darkblue}{rgb}{0,0,0.7}
\definecolor{darkred}{rgb}{0.7,0,0}

\usepackage[colorlinks,linkcolor=darkred, citecolor=darkblue, pagebackref=false,pdftex]{hyperref}

\newtheorem{theorem}{Theorem}

\newtheorem{proposition}[theorem]{Proposition}

\newtheorem{remark}[theorem]{Remark}

\renewcommand*{\dot}[1]{\accentset{\mbox{\large .}}{#1}}

\begin{document}

\title[Stabilization of affine systems with polytopic control value sets]{Stabilization of affine systems with polytopic control value sets}

\author[H. Leyva]{Horacio Leyva}
\email{horacio.leyva@unison.mx}
\address{Departamento de Matem\'aticas, Universidad de Sonora, M\'exico}

\author[B. Aguirre-Hern\'andez]{Baltazar Aguirre-Hern\'andez}
\email{bahe@xanum.uam.mx}
\address{Departamento de Matem\'aticas, Universidad Aut\'onoma Metropolitana Iztapalapa, M\'exico}

\author[J. F. Espinoza]{Jes\'us F. Espinoza}
\email{jesusfrancisco.espinoza@unison.mx}
\address{Departamento de Matem\'aticas, Universidad de Sonora, M\'exico}

\begin{abstract}
The objective of this document is to design continuous feedback controls for global asymptotic stabilization (GAS) of affine systems, with control restricted to a compact and convex set (CVS). This stabilization problem is solved based on a design of a feedback function restricted to the hyperbox and obtained by means of the CLF theory. By ``normalizing'' this feedback, the continuous stabilizer restricted to CVS is obtained.
\end{abstract} 

\maketitle

\section{Introduction}
Consider the multiple input continuous-time affine system
\begin{equation}\label{equation:affine_system}
    \dot{x} = f(x) + G(x)u,
\end{equation}
where $x\in\mathbb{R}^{n}$, $f, g_{i} : \mathbb{R}^{n} \rightarrow \mathbb{R}^{n}$, for $i=1,\ldots,m$, are $\mathcal{C}^{s}(\mathbb{R}^{n})$ vector fields ($s \geq 0$), $g_{i}(x)$ are the columns of the matrix $G(x)$, and the \textit{control value set} (CVS) is a bounded and convex subset of $\mathbb{R}^m$. Without loss of generality, we shall assume that $f(0)=0$. Such CVS will be required to be a sublevel set 
\[
U_{\phi}=\left\{ u \in \mathbb{R}^{m} \mid \phi(u) \leq 1 \right\}, 
\]  
where $\phi : \mathbb{R}^{m} \rightarrow \mathbb{R}_{+}$ is a convex and positively homogeneous function, that is, $\phi(ru) = r \phi(u)$ for any real number $r \geq 0$; in particular, $\partial U_{\phi }$ is given by the level set $\left\{ u \in \mathbb{R}^{m} \mid \phi(u) = 1 \right\}$. We will assume that the set $U_{\phi} \subset \mathbb{R}^{m}$ is compact and convex with $0 \in \mathrm{int}\, U_\phi$.

Is well known the usefulness of discontinuous controls in system stabilization, mainly to obtain robustness and stabilization in finite time, see \cite{Smirnov:1996}. However, discontinuous controls lead to non-modeled instabilities (such as ``chattering'', see \cite{Aguilar:2005}), then in this work we return to a continuous control design, with robustness properties, which can be used to stabilize affine systems with different CVS.

In order to obtain smooth stabilization, we consider the set of admissible feedback control functions $\mathcal{U}_{\phi}$ defined by
\[
\mathcal{U}_{\phi} := \left\{ u : \mathbb{R}^{n} \rightarrow U_{\phi } \mid u(x)\ \mbox{is continuous} \right\}.
\]

The main objective of this article is to address the problem of global asymptotic stabilization (GAS) of affine systems by means of an admissible feedback control $u(x)$.

Given any convex bounded CVS $U$, we seek to obtain continuous feedback control laws $u\left( x \right) \in U$ that stabilize nonlinear systems of type (\ref{equation:affine_system}). The relevance of this open problem was stated in \cite{Curtis}: ``Find universal formulas for CLF stabilization, for general (convex) control-value sets $U$''. To address this important problem, a review of the work carried out to date suggests that it is convenient to separate the bounded and convex sets $U$ in two classes, the sets $U$ with smooth boundary and those with a non-smooth boundary. In the second one class, we can find the polytopes, whose boundary $\partial U$ is piecewise linear.
% [{\color{red} AGREGAR REFERENCIAS ADICIONALES}]. 
% R. T. ROCKAFELLAR, Convex Analysis, Princeton University Press, Jan. 1997. p. 174
% Global CLF stabilization of systems with respect to a hyperbox, allowing the null-control input in its boundary (positive controls), in 53rd IEEE Conference on Decision and Control, IEEE, dec 2014, pp. 3107–3112.

Since Artstein's theorem (see \cite{ARTSTEIN19831163}) is valid in any bounded and convex CVS $U$, under the assumption that a ``control Lyapunov function'' (CLF) is known, we will approach the stabilization problem according to the line of work established in \cite{ARTSTEIN19831163, SoliDaun2014, Lin-Sontag:1995}. In the works \cite{Saperstone1971, SolisDaun2013, SolisDaun2015}, for a CVS $U\subset \mathbb{R}^{m}$ with a smooth boundary $\partial U$, studies were presented that addressed the stabilization problem using the CLF theory. %On the other hand, in \cite{Leyva2013, SOLISDAUN201111042, SoliDaun2014} some results were presented for the case where the CVS $U$ is a polytope. {\color{red} ?`C\'omo se relacionan con el presente trabajo?}
In the case of CVS a polytope, \cite{SOLISDAUN201111042, SoliDaun2014} shows the existence of an optimal feedback control for system (\ref{equation:affine_system}) that takes values at the vertices of the polytope and an explicit formula for it is obtained. In the case of the CVS represented by an asymmetric hyperbox, in \cite{Leyva2013} a continuous and explicit feedback function is presented to globally stabilize the system (\ref{equation:affine_system}).

In \cite[Formula 27]{Leyva2013} an explicit and decentralized design of admissible feedback controls $u^{\varepsilon }(x)$ restricted to a hyperbox $H \subset \mathbb{R}^{m}$ is presented, so that the proposed family of continuous controllers $u^{\varepsilon}(x)$ approaches the control that optimizes the \textit{robust stability margin}. In this paper we extend the hyperbox constrained continuous stabilizer design to other sets of control values, including the polytope case. With the exception of the hyperbox, i.e., a rectangular parallelepiped whose faces are each one perpendicular to some of the basis vectors, in the current literature there are no designs of continuous stabilizers restricted to a polytope.

% To obtain a continuous stabilizer restricted to a convex set $U\subset H$, with $0\in \mathrm{int}\, U$, it is required to represent $U$ as a level set:
% \begin{equation*}
% U_{\phi }:=\left\{ u\in \mathbb{R}^{m}:\phi (u)\leq 1\right\} ,
% \end{equation*}%
% where $\phi :\mathbb{R}^{m}\rightarrow \mathbb{R}_{+}$ is a convex and positively homogeneous function ($\phi (ru)=r\phi (u)$, for any real $r \geq 0 $); in particular, $\partial U_\phi = \{u \in \mathrm{R}^m \mid \phi(u) = 1\}$. We will assume that the set $U_{\phi}\subset \mathbb{R}^{m}$ is
% compact and convex with $0\in \mathrm{int}\, U_{\phi}$.

\section{Main results}

Given the system (\ref{equation:affine_system}) with CVS the hyperbox $H$, by means of a Lyapunov function $V(x)$, a design of admissible feedback functions is explicitly presented in \cite{Leyva2013}, with the property of being continuous, sub-optimal and decentralized. %How to extend this design to other convex CVS?

As was mentioned above, the study of convex sets has been divided in the literature into two large groups: the convex sets with smooth boundary (strictly convex) and polytopes. In \cite{Rockafellar1997} some concepts of convexity handled in this work can be consulted. In general, the problem of stabilizing system (\ref{equation:affine_system}) is strongly related to the particular characteristics of the CVS $U$, such as the smoothness of its boundary $\partial U$. In the literature of CLF theory there are stabilization studies of affine systems (\ref{equation:affine_system}), for a CVS $U$ bounded and strictly convex, with a smooth $\partial U$ boundary, articles \cite{Malisoff2000}, \cite{SolisDaun2013}, \cite{SolisDaun2015}, \cite{SolisDaun2010} and \cite{SONTAG1989117} correspond to this case.

%The works \cite{SONTAG1989117}, \cite{Malisoff2000}, \cite{SolisDaun2010}, \cite{SolisDaun2013} and \cite{SolisDaun2015} can be seen. On the other hand, there are few studies on the synthesis problem with control restricted to sets with non-smooth border (polytopes), some are \cite{Curtis}, \cite{SOLISDAUN201111042}, \cite{Leyva2013} and \cite{Leyva2014}.

%In \cite{HoracioLeyva2009} a design of non-negative control functions is presented for the stabilization problem of (\ref{equation:affine_system}), which corresponds to the $0 \in \partial U$ case. The case of multiple entry, with $U$ the hyperbox and $0 \in \partial U$, is studied at work \cite{Leyva2014}.

In general, in the stabilization problem of the affine system (\ref{equation:affine_system}) with a compact and convex set, is necessary a further work in finding explicitly admissible functions $u(x)$, with properties of smoothness and robustness.

% En general, en el problema de estabilizaci\'{o}n de (\ref{equation:affine_system}) con CVS convexo y compacto, siguen faltando funciones admisibles $u(x)$ presentadas expl\'{\i}citamente, en forma descentralizada, que incluyan caracter\'{\i}sticas de suavidad y robustez.

\subsection{Types of CVS \texorpdfstring{$U_{\phi}$}{U-phi}.}

Some results about convexity theory, considered implicitly in the development of this work can be found in \cite{Rockafellar1997}.

Examples of the support function $\phi$ for a non-empty compact convex set $U_{\phi}$, are the following:
\begin{itemize}
    \item $\phi_{1}(u) = L^{T}\left\vert u \right\vert$, where $L^{T} = (l_{1}, l_{2}, \ldots, l_{m})$, with $l_{i}$ positive constants.
    \item $\phi_{2}(u) = u^{T}Qu$, where $Q \in \mathbb{R}^{m\times m}$ is a positive definite matrix.
    \item $\phi_{3}(u) = \max_{i=1,\ldots, k}\{v_{i}^{T}u\}$, for $v_{1}, v_{2},\ldots, v_{k} \in \mathbb{R}^{m}$ non-zero vectors.
\end{itemize}

% \begin{equation} %\label{equation:support}
%     \phi_{1}(u) = L^{T}\left\vert u \right\vert, \quad \phi_{2}(u) = u^{T}Qu \quad\mbox{and}\quad \phi_{3}(u) = \max_{i=1,\ldots, k}\{v_{i}^{T}u\}
% \end{equation}
% where $L^{T} = (l_{1}, l_{2}, \ldots, l_{m})$, with $l_{i}$ positive constants, $\{v_{1}, v_{2},\ldots, v_{k}\} \in \mathbb{R}^{m}$
% represent non-zero vectors and $Q \in \mathbb{R}^{m\times m}$ is a positive definite matrix. 
For each such a support functions $\phi_{i}$, we assume that $0 \in \mathrm{int\,} U_{\phi_{i}}$, so that $\phi_{i}(u) = 0$ only if $u = 0$. The sets $U_{\phi_{i}}$ represented by
\begin{equation*} %\label{equation:U-phi}
    U_{\phi_{i}} := \{u \in \mathbb{R}^{m} \mid \phi_{i}(u) \leq 1 \}, \quad i = 1, 2, 3,
\end{equation*}
are compact and convex subsets of $\mathbb{R}^{m}$, where $U_{\phi_{1}}$ is a particular polytope with $2^m$ vertices and symmetric centered at the origin, $U_{\phi_{2}}$ is an ellipsoid also centered at the origin.% and $U_{\phi_{3}}$ is a bounded polyhedral set with $0 \in \mathrm{int\,} U$.

% An important class of compact convex sets is the class of convex polytopes, see \cite{SoliDaun2014}. 
For every convex polytope $P \subset \mathbb{R}^{m}$, with $0 \in \mathrm{int\,}P$, there are vectors $\{v_{1}, v_{2}, \ldots, v_{k}\} \in \mathbb{R}^{m}$, such that by means of the continuous non-negative function and piecewise linear (see \cite[Theorem 1.1]{Ziegler1995}, \cite[p. 174]{Rockafellar1997}),
\begin{equation*}
    \phi(u) = \max_{i=1,\ldots,k} \{v_{i}^{T}u \}
\end{equation*}
so that we can represent the polytope $P$ as
\begin{equation*} %\label{equation:polytope}
    P := \{u \in \mathbb{R}^{m} \mid \phi(u) \leq 1 \},
\end{equation*}
which we can denote as $P_{\phi}$. In \cite{SOLISDAUN201111042} and \cite{SoliDaun2014} polytopes are considered as CVS, giving rise to the corresponding set of admissible feedback functions $\mathcal{U}_{\phi}$. Currently there are no continuous stabilizers restricted to polytopes.

% How to design an admissible feedback function $u(x)$ to stabilize the solutions of system (\ref{equation:affine_system}) at equilibrium $x = 0$?

\subsection{Lyapunov function and Artstein's theorem}
The admissible stabilizer is obtained based on Artstein's theorem, see \cite{ARTSTEIN19831163}. Suppose that system (\ref{equation:affine_system}) is regular and $U_{\phi} \subset \mathbb{R}^{m}$ is a CVS. There is a smooth Lyapunov function $V(x)$ if there is a continuous control $u(x)$, except possibly at $x=0$, restricted to taking values in $U_{\phi}$, which generates the stabilization of the system (\ref{equation:affine_system}).

Given the system (\ref{equation:affine_system}) and the CVS $U_{\phi}$, to obtain an admissible stabilizer $u(x) \in \mathcal{U}_{\phi}$, two conditions must be met: the CLF condition and the SCP property.

\subsubsection*{The CLF condition} A non-negative function $V:\mathbb{R}^{n} \rightarrow \mathbb{R}_{+}$ is called the \emph{control Lyapunov function} (CLF), with respect to the system (\ref{equation:affine_system}) and the constraint $U_{\phi}$, if it happens that
\begin{equation} \label{equation:SCP_eq}
    \min\limits_{u\in U_{\phi}} \{a(x) + \beta(x) \cdot u\} < 0, \quad\mbox{ for all } x \neq 0,
\end{equation}
where
\begin{equation}\label{equation:CLF}
    a(x) := L_{f} V(x) \quad \& \quad \beta(x) := (\beta_{1}(x), \ldots, \beta_{m}(x)), \quad \mathrm{with\ } \beta_{i}(x) := L_{g_{i}} V(x),\ \, i=1,\ldots,m.
\end{equation}
%\begin{equation} \label{equation:CLF}
%    \min\limits_{u \in U_{\phi}} \{\nabla V(x) \cdot (f(x) + G(x)u)\} < 0 \quad \mbox{ for all } %\ x \neq 0.
%\end{equation}
This inequality means that there is an optimal stabilizer $\omega(x)$, which is not admissible because it is discontinuous; if the set $U_{\phi}$ is a polytope, the function $\omega(x)$ takes values only at the vertices of the polytope (see \cite{Leyva2014}, \cite{Leyva2013}, \cite{SOLISDAUN201111042} and \cite{SoliDaun2014}), and represents the control that gives the system the ``best stabilization rate'', according to the derivative of the Lyapunov function $V(x)$. A relevant purpose here is to find
a continuous function that approaches $\omega(x)$, without losing the previous inequality.

In \cite{ARTSTEIN19831163}, Zvi Artstein proved that the existence of a continuous stabilizing feedback control taking values in a convex CVS $U\subset\mathbb{R}^{m}$ is equivalent to the existence of a control Lyapunov function (CLF).

\subsubsection*{The SCP property} The existence of a continuous stabilizer at the origin is ensured by the \emph{small control property} (SCP): For each $\varepsilon >0$ there exists $\delta >0$ such that we have the inequality
\[ a(x) + \beta(x) \cdot u < 0, \mbox{ for all } x \neq 0, \]
for $u$ with $\Vert u \Vert_{U_{\phi}} < \varepsilon$, provided that $0 < \Vert x \Vert < \delta$, $a(x)$ and $\beta(x)$ as were defined above for (\ref{equation:SCP_eq}).

%We can rewrite the inequality (\ref{equation:CLF}) as
%\begin{equation} \label{equation:SCP_eq}
%    \min\limits_{u\in U_{\phi}} \{a(x) + \beta(x) \cdot u\} < 0, \quad\mbox{ for all } x \neq 0,
%\end{equation}
%where
%\begin{equation*}
%    a(x) := L_{f} V(x) \quad \& \quad \beta(x) := (\beta_{1}(x), \ldots, \beta_{m}(x)), \quad \mathrm{with\ } \beta_{i}(x) := L_{g_{i}} V(x),\ \, i=1,\ldots,m.
%\end{equation*}

We consider that the control value set is given by the hyperbox
\begin{equation*} %\label{equation:hyperbox}
    H := [-r_{1}^{-}, r_{1}^{+}] \times \cdots \times [-r_{m}^{-}, r_{m}^{+}] \subset \mathbb{R}^{m}, r_{i}^{-}, r_{i}^{+} > 0,
\end{equation*}
which can also be represented as
\begin{equation*} %\label{equation:hyperbox_eq}
    H := \{u \in \mathbb{R}^{m} \mid \max\limits_{i=1,...,m} \{\vert u_{i}\vert/r_{i}\} \leq 1\}
\end{equation*}
where $r_{i}$ for $i=1, 2, \ldots, m$, is defined as
\begin{equation*}
    r_{i}(b) := \begin{cases}
    r_{i}^{+} & if\ b \geq 0, \\
    r_{i}^{-} & if\ b \leq 0.
    \end{cases}
\end{equation*}

Therefore, for compact sets $H$ and $U_{\phi}$, with $0 \in \mathrm{int}\, U_{\phi} \subset H \subset \mathbb{R}^{m}$, it happens that
\begin{equation*}
    \min\limits_{u\in H} dV/dt \leq \min\limits_{u\in U_{\phi}} dV/dt,
\end{equation*}
and the CLF condition and SCP property remain when changing the CVS $U_{\phi}$ to $H$.

\subsection{An explicit feedback control formula with respect to a hyperbox.} The $\varepsilon$-parameterized design ($\varepsilon > 0$) of the family of feedback control functions $u^{\varepsilon}(x)$ presented in \cite[Theorem 14]{Leyva2013} is considered, which was obtained by means of the Artstein's theorem with the hyperbox $H$ as CVS. % {\color{red}For a polytope $P \subset H$, the class of systems that are GAS by continuous control functions taking values in the control value set $P$ is a proper subset of the class of systems that are GAS by controllers restricted to the hyperbox $H$.}

%{\color{red}The control feedback function presented in \cite{Leyva2013} consists of a decentralized design where the inputs $u_{i}(x)$ and $u_{j}(x)$, with $i\neq j$, assume values independently.} 
This feedback function $u^{\varepsilon}(x)$ is admissible with the hyperbox $H$, explicitly given, decentralized and sub-optimal, defined as follows:
\begin{equation} \label{equation:feedback}
    u^{\varepsilon}(x) := (u_{1}^{\varepsilon}(x), \ldots, u_{m}^{\varepsilon}(x))
\end{equation}
with
\begin{equation*}
    u_{i}^{\varepsilon}(x) = \varrho_{i}^{\varepsilon}(a(x), \vert \beta \vert r(x))\, \overline{\omega}_{i}(x),
\end{equation*}
where $\overline{\omega}(x)$ is the best rate control sharing the scheme of $\min\limits_{u \in H} dV/dt$, with the non-negative function $\vert \beta \vert r := \vert \beta_{1} \vert r_{1} + \cdots + \vert \beta_{m} \vert r_{m}$. The function $\varrho_{i}^{\varepsilon} : \mathbb{R} \times [0, \infty] \rightarrow \mathbb{R}$ is defined by
\begin{equation*} %\label{equation:varrho}
    \varrho_{i}^{\varepsilon}(a, \beta) =
    \begin{cases}
    1 - \left(1 - \dfrac{\vert a \vert + a}{2\, \vert \beta \vert r}\, \dfrac{\vert \beta_{i} \vert r_{i}}{\vert \beta \vert r}\right) \exp\left(\tau_{i}^{\varepsilon} \dfrac{ \vert \beta_{i} \vert r_{i}}{\vert \beta \vert r}\right) & if\ \vert \beta_{i} \vert r_{i} > 0, \\
    0 & if\ \vert \beta_{i} \vert r_{i}=0,
    \end{cases}
\end{equation*}
and $\tau_{i}^{\varepsilon}(x)$ is a non-positive function defined as
\begin{equation} \label{equation:tau}
    \tau_{i}^{\varepsilon}(x) = 
    \begin{cases}
    m \dfrac{\ln(\lambda(x))}{\lambda(x)} - \varepsilon \vert \beta_{i} \vert r_{i} & if\ \vert \beta \vert r > 0, \\
    0 & if\ \vert \beta \vert r = 0,
    \end{cases}
\end{equation}
for $i=1, \ldots, m$, where $\lambda(x) = 1 - \frac{1}{2}(\vert a(x) \vert + a(x))/ \vert \beta \vert r$ and $\varepsilon > 0$ is a tuning
parameter.

The control (\ref{equation:feedback}) - (\ref{equation:tau}) is continuous with respect to $x$, since the regulating function $\varrho^{\varepsilon}(a, \beta)$ cancels the discontinuities of the optimal stabilizer $\overline{\omega}(x)$.

\section{An explicit feedback control formula regarding CVS \texorpdfstring{$U_{\phi} \subset H$}{U C H}} \label{section:explicit_feedback}

Let $\phi:\mathbb{R}^{m} \rightarrow \mathbb{R}_{+}$ be a positively homogeneous and convex function, so that the following compact set $U_{\phi}$ can be defined as the level set
\begin{equation*} %\label{equation:level-set}
    U_{\phi} = \{ u \in \mathbb{R}^{m} \mid \phi(u) \leq 1\}, 
\end{equation*}

Let's see now, how to stabilize systems of type (\ref{equation:affine_system}) with admissible feedback functions restricted to $U_{\phi}$. Consider the continuous feedback controls in a decentralized way $u^{\varepsilon}(x)$, given by (\ref{equation:feedback}) - (\ref{equation:tau}). The main idea is to extend the feedback function $u^{\varepsilon}(x)$ restricted to the hyperbox $H$, so that the feedback function $u_{\phi}^{\varepsilon}(x)$, restricted to the new CVS $U_{\phi} \subset H$. %, in such a way that it preserves the same properties; that is, they represent decentralized, sub-optimal and admissible stabilizers.

Once the set $U_{\phi}$ has been defined, a hyperbox $H$ such that $U_{\phi}\subset H$ is chosen (the smallest possible). Let $M$ be such that
\begin{equation*} %\label{equation:M}
    M := \max\limits_{H} \phi(u),
\end{equation*}
thus
\[ 0 \leq \min_{H} \phi(u) \leq \phi(u) \leq \max_{H} \phi(u) = M,\]
therefore, for the case $1 \leq \phi(u)$ and for any non-negative function $a(x)$, we have
\[ \dfrac{a(x)}{M} \leq \dfrac{a(x)}{\phi(u)} \leq a(x). \]

Now, consider the affine system
\begin{equation} \label{equation:M_affine_system}
    \dot{x} = \dfrac{1}{M} f(x) + g_{1}(x) w_{1} + \cdots + g_{m}(x)w_{m},
\end{equation}
with control $w = (w_{1}, \ldots, w_{m})^{T} \in U_{\phi}$. Considering the admissible feedback function $u^{\varepsilon}(x) \in H$
given by (\ref{equation:feedback}) - (\ref{equation:tau}), we propose the following feedback function $v:\mathbb{R}^{n} \rightarrow U_{\phi}$, given by
\begin{equation} \label{equation:M_feedback}
    w_{\phi}^{\varepsilon}(x) = 
    \begin{cases}
    u^{\varepsilon}(x) & if\ \phi(u^{\varepsilon}(x)) \leq 1, \\
    \dfrac{1}{\phi(u^{\varepsilon}(x))} u^{\varepsilon}(x) & if\ \phi(u^{\varepsilon}(x)) \geq 1.
    \end{cases}
\end{equation}

\begin{proposition}
If the function $V(x)$ is CLF and satisfies the SCP with respect to the affine system (\ref{equation:affine_system}) with CVS the hyperbox $H$, then the feedback function $w_{\phi}^{\varepsilon}(x)$ given by (\ref{equation:M_feedback}) is admissible and the feedback system (\ref{equation:M_affine_system})-(\ref{equation:M_feedback}) is GAS.
\end{proposition}

\begin{proof}
The continuity of $w^{\varepsilon}(x)$ is inherited from the continuity of $u^{\varepsilon}(x)$, see \cite[Prop. 12, Theorem 14]{Leyva2013}. For the case $\phi(u^{\varepsilon}(x)) \leq 1$ it is immediate, since $w^{\varepsilon}(x) = u^{\varepsilon}(x)$. If $\phi(u^{\varepsilon}(x)) \geq 1$, then
\begin{equation*}
    \dfrac{1}{\phi(u^{\varepsilon}(x))} u_{i}^{\varepsilon}(x),
\end{equation*}
for $i = 1,\ldots , m$, is a quotient of continuous functions, in fact, they are the components of the vector function $\dfrac{1}{\phi(u^{\varepsilon}(x))} u^{\varepsilon}(x)$.

It is satisfied that $w^{\varepsilon}(x) \in U_{\phi}$, since for the case $\phi(u^{\varepsilon}(x)) \geq 1$ we have that
\[ \phi\left(\dfrac{1}{\phi(u^{\varepsilon}(x))} u^{\varepsilon}(x)\right) = \dfrac{1}{\phi(u^{\varepsilon}(x))} \phi(u^{\varepsilon}(x)) = 1, \]
since $\phi$ is positively homogeneous.

Next, we prove that the feedback system (\ref{equation:M_affine_system}) - (\ref{equation:M_feedback}) is globally asymptotically stable. With the design $u^{\varepsilon}(x) \in H$ given by (\ref{equation:feedback}) - (\ref{equation:tau}), we have
\begin{equation*}
    a(x) + b_{1} u_{1}^{\varepsilon}(x) + \cdots + b_{m}u_{m}^{\varepsilon}(x) < 0 \quad \mbox{ for all } x \neq 0,
\end{equation*}
such that, for the case $\phi(u) \geq 1$, with $a(x) \geq 0$, we have
\begin{equation*}
    \frac{1}{M} a(x) \leq \frac{1}{\phi(u)} a(x) \leq a(x).
\end{equation*}
Therefore,
\begin{equation*}
    \frac{1}{M} a(x) + b_{1}\frac{1}{\phi(u)} u_{1}(x) + \cdots + b_{m}\frac{1}{\phi(u)} u_{m}(x) <0 \quad \mbox{ for all } x \neq 0,
\end{equation*}
we conclude that the feedback system (\ref{equation:M_affine_system}) - (\ref{equation:M_feedback}) is globally asymptotically stable.
\end{proof}

%{\color{red} \textbf{[PENDIENTE]} Proposición 2 (modicado). Si el sistema afín  con CVS la hipercaja H, la función es CLF y satisface la propiedad SCP, entonces el sistema retroalimentado GAS.}

By \cite[Formula (27)]{Leyva2013}, we have the admissible feedback $u(x) = (u_{1}(x),\ldots , u_{m}(x))^{T}\in H$ with coordinate functions $u_{i}(x) = \rho _{i}^{\varepsilon }\left( x\right) \omega _{i}\left( x\right)$ and $\varepsilon > 0$ a tunning parameter, with a rescaling vector $\rho^{\varepsilon }\left( x\right) =\left( \rho _{1}^{\varepsilon }\left(x\right) ,\ldots ,\rho _{m}^{\varepsilon }\left( x\right) \right) $, and $\omega (x)=(\omega _{1}(x),...,\omega _{m}(x))^{T}$ being the CLF-optimal solution of (\ref{equation:CLF}), for $u \in H$. So that, the formula (\ref{equation:M_feedback}) has components as follows,
\begin{align*}
w_{i}^{\varepsilon} \left(x\right) = 
\left\{ 
    \begin{array}{ccc}
    \rho_{i}^{\varepsilon} \left(x\right) \omega_{i} \left(x\right)  & \mbox{if } \phi \left(u\left(x\right)\right) \leq 1, \\
    &  \\ 
    \dfrac{1}{\phi \left(u^{\varepsilon} \left(x\right)\right)} \rho_{i}^{\varepsilon} \left(x\right) \omega_{i} \left(x\right) & \mbox{if } \phi \left(u\left(x\right)\right) \geq 1.
    \end{array}
\right.
\end{align*}

From \cite[Theorem 14]{Leyva2013}, if $\beta _{i}(x)\neq 0$, $i=1,2,\ldots ,m$, then $\lim\limits_{\varepsilon \rightarrow \infty} u^{\varepsilon}(x) = \omega(x)$, therefore the control (\ref{equation:M_feedback}) satisfies that,
\begin{equation*}
\lim\limits_{\varepsilon \rightarrow \infty}v^{\varepsilon}\left(x\right)
=\dfrac{1}{\phi \left(\omega(x) \right)} \omega(x)\in \partial U_{\phi },
\end{equation*}
since $\phi \left( \dfrac{1}{\phi \left( \omega (x)\right) }\omega (x)\right) =\dfrac{1}{\phi \left( \omega (x)\right) }\phi \left( \omega (x)\right) =1$. Then, if $u^{\varepsilon }\left( x\right) \in
\partial H$, it follows that $w^{\varepsilon }\left( x\right) \in \partial U_{\phi}$.

\begin{remark} Given an open-loop unstable system (i.e., $a\left( x\right) \geq 0$), from (\ref{equation:SCP_eq}) with control (\ref{equation:M_feedback}) we have that the global instability of the system with feedback can be represented by the inequality
\[ \frac{1}{k} a(x) + \beta(x) w_{\phi}^{\varepsilon} (x) < 0 \text{, for all } x \neq 0 \text{ and for } k \geq M \geq 1,
\]
so that, the admissible formula (\ref{equation:M_feedback}) presents a tradeoff: the magnitude of the constant $M \geq 1$ is directly proportional to the size of the set $H\backslash U_{\phi }$, on such way that decreases the size of the instability $\dfrac{1}{k}a\left( x\right)$ in order to hold the above inequality.
\end{remark}

\section{Example} \label{section:example}
Let us consider the affine system (\ref{equation:affine_system}) with $m=2$ and CVS given by the triangle $T$ defined by $T = \mathrm{conv}\{v_0 = (0, -2), v_1 = (\sqrt{3}, 1), v_2 = (-\sqrt{3},1)\}$, depicted in Figure \ref{Figure:T}. \bigskip
% \begin{equation*}
%     T = \mathrm{conv}\left\{
%     \begin{pmatrix}
%     \sqrt{3} \\ 1
%     \end{pmatrix}, 
%     \begin{pmatrix}
%     -\sqrt{3} \\ 1
%     \end{pmatrix},
%     \begin{pmatrix}
%     0 \\ -2
%     \end{pmatrix}
%     \right\}.
% \end{equation*}

\begin{figure}[H] 
\centering
  \begin{subfigure}{0.33\textwidth}
    \centering
    \includegraphics[width=0.9\linewidth]{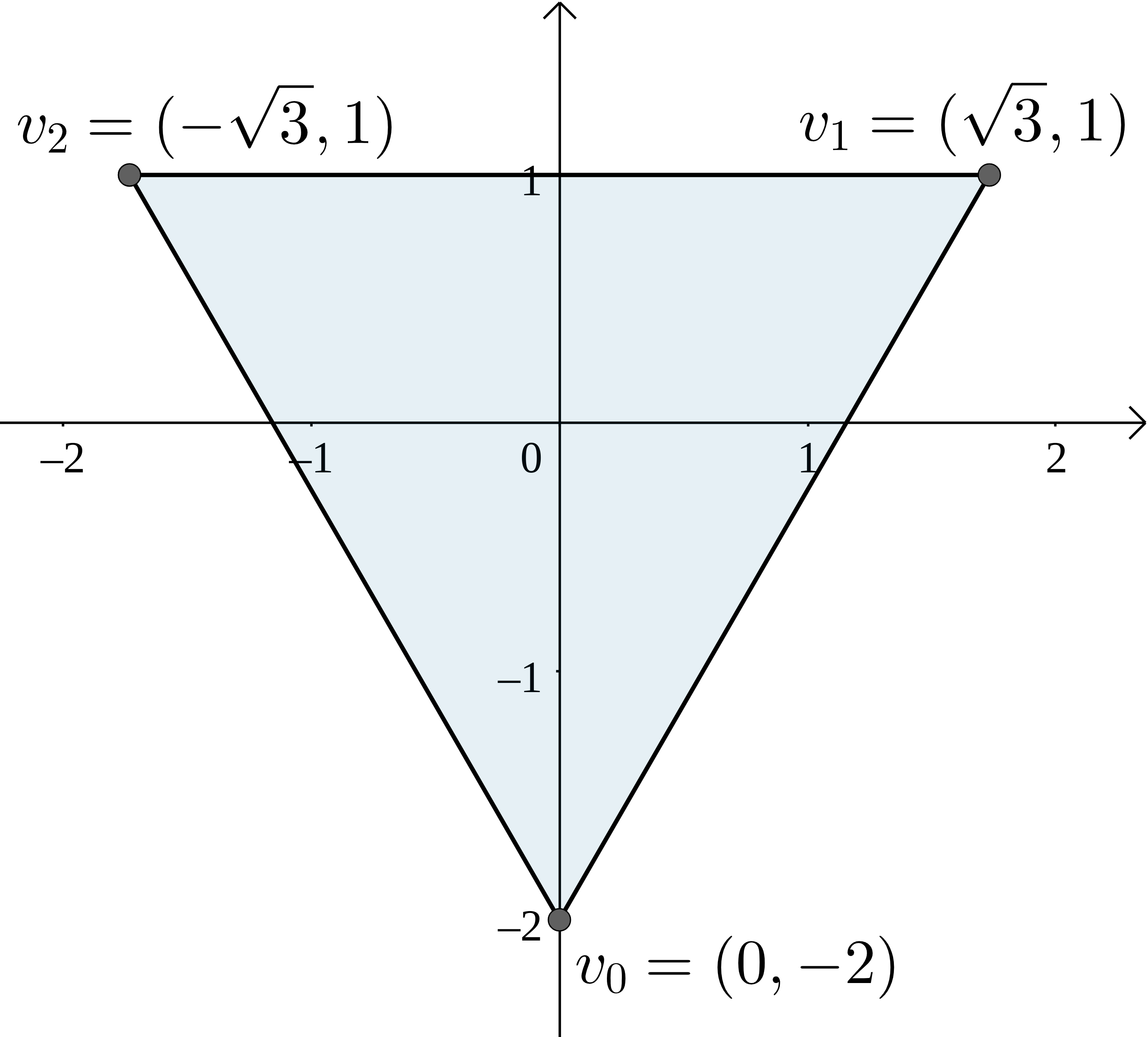}
    \subcaption{$T = \mathrm{conv}\{v_0, v_1, v_2\}.$}
    \label{Figure:T}
    \end{subfigure}\hfill
  \begin{subfigure}{0.33\textwidth}
    \centering
    \includegraphics[width=0.9\linewidth]{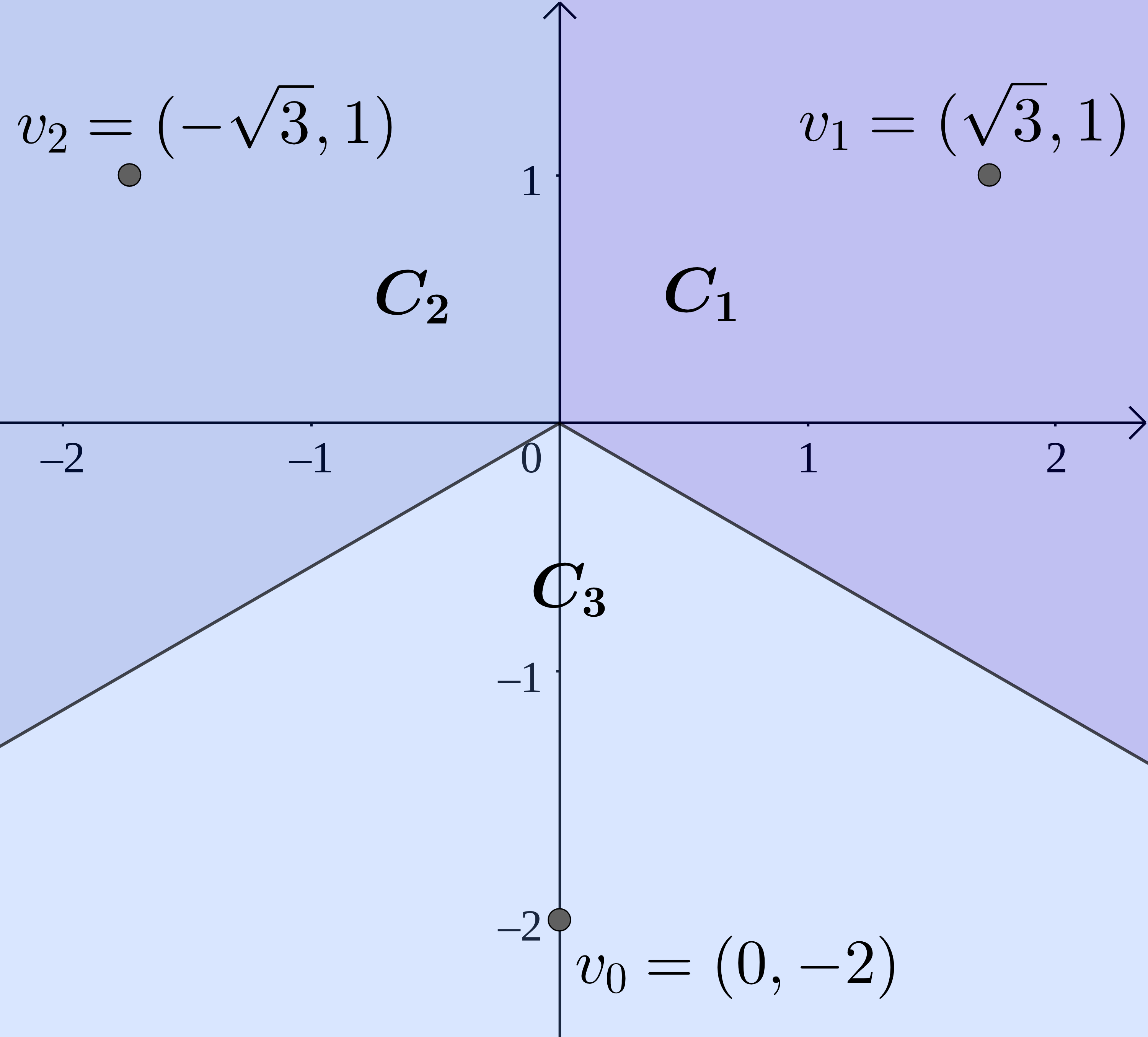}
    \subcaption{$\mathbb{R}^2 = C_1 \cup C_2 \cup C_3.$}
    \label{Figure:C}
  \end{subfigure}
  \begin{subfigure}{0.33\textwidth}
    \centering
    \includegraphics[width=0.9\linewidth]{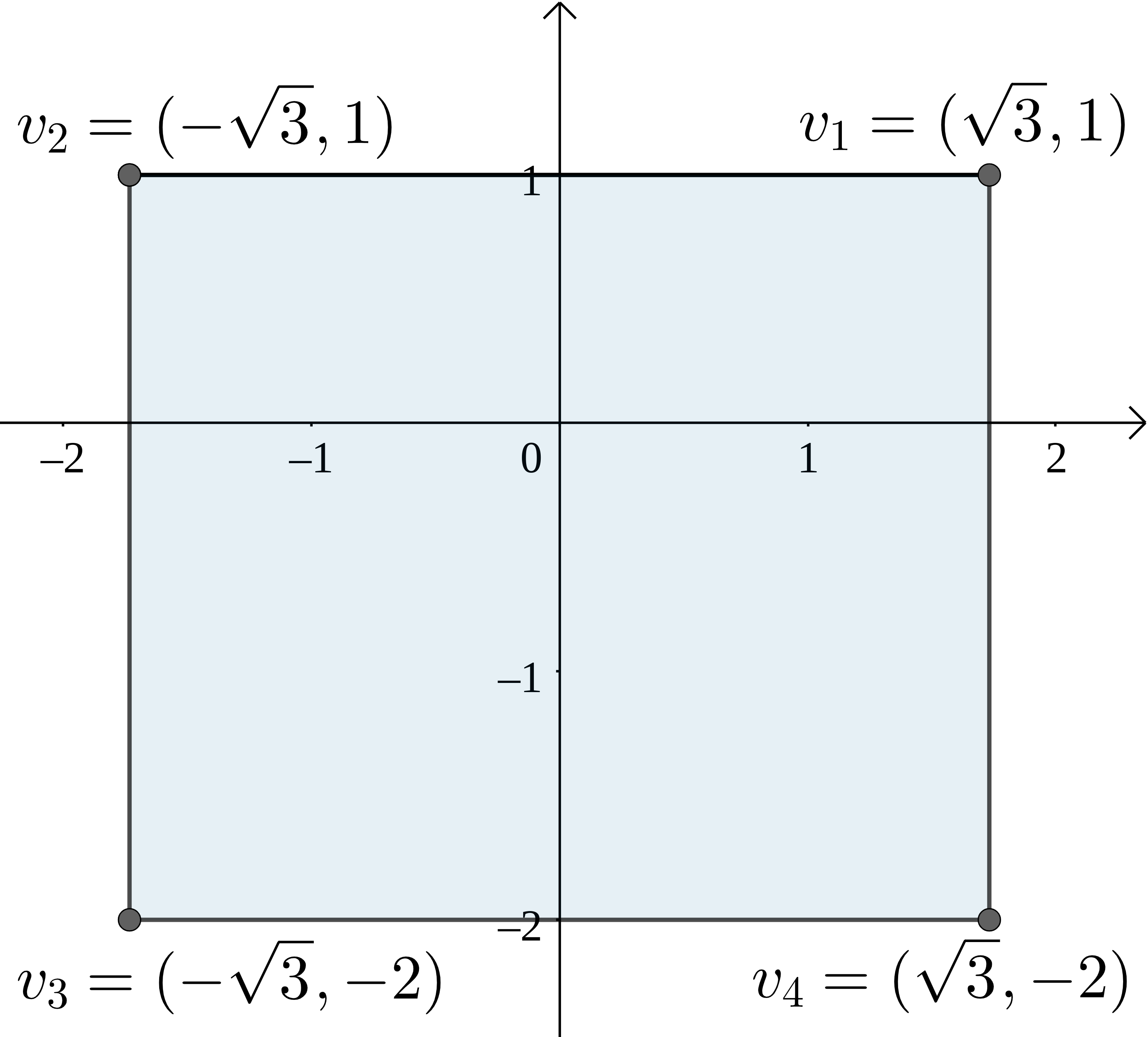}
    \subcaption{$H = \mathrm{conv}\{v_1, v_2, v_3, v_4\}.$}
    \label{Figure:H}
  \end{subfigure}
  \caption{}
\end{figure}

Suppose we know a Lyapunov function $V:\mathbb{R}^{n}\rightarrow \mathbb{R}_{+}$, so that the CLF and SCP properties hold:
\[ \min\limits_{u \in H} \dot{V} = a(x) + b_{1} \omega_{1} + b_{2} \omega_{2} < 0 \mbox{ for all } x \neq 0, \]
where $ a(x) := L_{f}V(x)$ and $b(x) := (b_{1}(x), b_{2}(x))$, with $b_{i}(x) := L_{g_{i}} V(x)$, $i=1,2$. Hence, from \cite[Formula 21]{SOLISDAUN201111042}, we obtain
\begin{equation*}
    \omega(b_1, b_2) = 
    \begin{cases}
    v_{1} = (\sqrt{3}, 1)    & \mbox{ if } (b_1, b_2) \in C_1, \\
    v_{2} = (-\sqrt{3}, 1)   & \mbox{ if } (b_1, b_2) \in C_2, \\
    v_{3} = (0, -2)          & \mbox{ if } (b_1, b_2) \in C_3, \\
    \end{cases}
\end{equation*}
where (see Figure \ref{Figure:C})
\[ C_{1} = \left\{ (b_{1}, b_{2}) \in \mathbb{R}^{2} \right|\left. b_{1} \geq 0, b_{2} \geq -\frac{b_{1}}{\sqrt{3}}\right\},\]
\[ C_{2} = \left\{ (b_{1}, b_{2}) \in \mathbb{R}^{2} \right|\left. b_{1} \leq 0, b_{2} \geq \frac{b_{1}}{\sqrt{3}} \right\},\]
and
\[ C_{3} = \left\{ (b_{1}, b_{2}) \in \mathbb{R}^{2} \right|\left. b_{2} \leq 0, \sqrt{3}b_{2} \leq b_{1} \leq -\sqrt{3} b_{2} \right\},\]
such that $\omega (b)=\left( \omega _{1},\omega _{2}\right) $ is constant on
each open polytopal cone $\mathrm{int}\, C_{i}$, and it is equal to the vertices of the
triangle $T$.

Instead, if the CVS is the hyperbox $H$, defined by
\[ H := \mathrm{conv} \{v_1 = (\sqrt{3}, 1), v_2 = (-\sqrt{3},1), v_{3} = (-\sqrt{3}, -2), v_{4} = (\sqrt{3}, -2)\}, \]
% \begin{equation*}
% \left\{ \bar{v}_{1}=\left( 
% \begin{array}{c}
% \sqrt{3} \\ 
% 1%
% \end{array}%
% \right) ,\ \bar{v}_{2}=\left( 
% \begin{array}{c}
% -\sqrt{3} \\ 
% 1%
% \end{array}%
% \right) ,\ \bar{v}_{3}=\left( 
% \begin{array}{c}
% -\sqrt{3} \\ 
% -2%
% \end{array}%
% \right) ,\ \bar{v}_{4}=\left( 
% \begin{array}{c}
% \sqrt{3} \\ 
% -2%
% \end{array}%
% \right) \right\} ,
% \end{equation*}%
equivalently, defined as $ H: = [-\sqrt{3}, \sqrt{3}] \times [-2, 1] = [-r_{1}^{-}, r_{1}^{+}] \times [-r_{2}^{-}, r_{2}^{+}] \subset \mathbb{R}^{2}$ (see Figure \ref{Figure:H}), then the optimal stabilizer (with the best rate) is $(\overline{\omega}_{1}, \overline{\omega}_{2}) \in H$, defined by the equality
\begin{eqnarray*}
\min\limits_{u \in H} \dot{V} &=& a(x) + b_1 \overline{\omega}_1 + b_2 \overline{\omega}_2 \\
&=& a(x) - (\vert b_1 \vert r_1 + \vert b_2 \vert r_2),
\end{eqnarray*}
consequently (see example in \cite{Leyva2013}),
\[
\overline{\omega}(b_1, b_2) = 
    \begin{cases}
    v_{1} & \mbox{ if } (b_1, b_2) \in \mathrm{cl}(\mathbb{R}^2_{++}), \\
    v_{2} & \mbox{ if } (b_1, b_2) \in \mathrm{cl}(\mathbb{R}^2_{-+}), \\
    v_{3} & \mbox{ if } (b_1, b_2) \in \mathrm{cl}(\mathbb{R}^2_{--}), \\
    v_{4} & \mbox{ if } (b_1, b_2) \in \mathrm{cl}(\mathbb{R}^2_{+-}). \\
    \end{cases}
\]
% \begin{equation*}
% \overline{\omega}\left( b_{1},b_{2}\right) =\left\{ 
% \begin{array}{c}
% \bar{v}_{1},\quad \text{if}\ b\in \mathbb{\bar{R}}_{--}^{2} \\ 
% \bar{v}_{2},\quad \text{if}\ b\in \mathbb{R}_{+-}^{2} \\ 
% \bar{v}_{3},\quad \text{if}\ b\in \mathbb{\bar{R}}_{++}^{2} \\ 
% \bar{v}_{4},\quad \text{if}\ b\in \mathbb{R}_{-+}^{2}%
% \end{array}%
% \right. \text{,}\ 
% \end{equation*}%
The CLF condition for system (\ref{equation:affine_system}) with CVS $H$ implies that $\min\limits_{u \in H} \dot{V} < 0$ for all $x \neq 0$, so that for $a(x) \geq 0$, the inequality
\[ 0 \leq \frac{a(x)}{\beta(x)} < 1, \]
is satisfied, or else
\[ 0 \leq \frac{a(x) + \vert a(x) \vert }{2 \beta(x)} < 1. \]
Hence, we define the following non-negative functions:
\[
\begin{aligned}
    r_i (b_i) &:= 
        \begin{cases}
            r_i^+ & \mbox{ if } b_i \geq 0 \\
            r_i^- & \mbox{ if } b_i \leq 0
        \end{cases}, i = 1, 2, \\
    \beta &:= \vert b_1 \vert r_1 + \vert b_2 \vert r_2 \\
    \vert a \vert + a &:= \vert L_{f} V(x) \vert + L_{f} V(x) \\
    \lambda(x) &:= 1 - \frac{1}{2} (\vert a(x) \vert + a(x)) / \beta(x)
\end{aligned}
\]
and a non-positive function defined as
\[
\tau _i^\varepsilon (x) = 
    \begin{cases}
        m \dfrac{\ln (\lambda(x))}{\lambda(x)} - \varepsilon \vert b_i \vert r_i & \mbox{ if } \beta > 0, \\
        0 & \mbox{ if } \beta = 0,
    \end{cases}
\]
with $\varepsilon >0$ is a tuning parameter. The function $\varrho_i^\varepsilon : \mathbb{R} \times [0, \infty] \rightarrow \mathbb{R}$ is defined by
\[
\varrho_i^\varepsilon (a, \beta) = 
    \begin{cases}
        1 - \left(1 - \dfrac{\vert a \vert + a}{2\beta}\, \dfrac{\vert b_i \vert r_i}{\beta}\right) \exp \left(\tau_i^\varepsilon \dfrac{\vert b_i \vert r_i}{\beta}\right) & \mbox{ if } \vert b_i \vert r_i >0, \\
        0 & \mbox{ if } \vert b_i \vert r_i = 0.
    \end{cases}
\]
This feedback function $u^\varepsilon (x)$ is admissible with the hyperbox $H$, explicitly given and sub-optimal, defined as follows:
\[ u^\varepsilon (x) := (u_1^\varepsilon (x), u_2^\varepsilon (x)) \]
with
\[ u_i^\varepsilon (x) = \varrho_i^\varepsilon (a(x), \beta(x)) \overline{\omega}_i (x),\ i=1,2. \]

In particular, for the affine system
\[
f(x_1, x_2) = 
\begin{pmatrix}
    \sqrt{3} \dfrac{x_2}{1 + x_2^2} \\ 
    \dfrac{x_1}{1 + x_1^2} + \dfrac{x_1^2}{1 + x_2^2}
\end{pmatrix},
g_1 = 
\begin{pmatrix}
    1 \\ 
    0
\end{pmatrix},
g_2 = 
\begin{pmatrix}
    0 \\
    1
\end{pmatrix},
\]
with rectangular CVS $H = [-\sqrt{3}, \sqrt{3}] \times [-2, 1]$, it is possible to generate (see formulas (29)-(31) in \cite{Lin-Sontag:1995}) the continuous control $u^{\varepsilon}(x) := (u_1^\varepsilon (x), u_2^\varepsilon (x))$, given by
\[ u_i^\varepsilon (x) = \varrho_i^\varepsilon (a(x), \beta(x))\, \overline{\omega}_i (x),\ i = 1, 2. \]

Suppose we know a Lyapunov function $V : \mathbb{R}^n \rightarrow \mathbb{R}_{+}$, so that the CLF and SCP properties hold. With
\[ V (x_1, x_2) = \dfrac{1}{2} \left(x_1^2 + x_2^2 \right), \]
then
\[ \dot{V} = x_1 \left(\sqrt{3} \dfrac{x_2}{1 + x_2^2} + u_1\right) + x_2 \left(\dfrac{x_1}{1 + x_1^2} + \dfrac{x_2^2}{1 + x_2^2} + u_2\right) \]
so that
\[ a(x) = x_1 \left(\sqrt{3} \dfrac{x_2}{1 + x_2^2}\right) + x_2 \left( \dfrac{x_1}{1 + x_1^2} + \dfrac{x_2^2}{1 + x_2^2} \right) \]
and $b_1 = x_1$, $b_2 = x_2$, therefore 
\[
\begin{aligned}
    r_1 (x_1) &:= \sqrt{3}, \\
    r_2 (x_2) &:= 
        \begin{cases}
            1 & \mbox{ if } x_2 \geq 0, \\
            2 & \mbox{ if } x_2 \leq 0,
        \end{cases} \\
    \beta &:= \sqrt{3}\vert x_1 \vert + r_2 \vert x_2 \vert.
\end{aligned}
\]
% \begin{eqnarray*}
% i)\qquad \qquad r_{1}\left( x_{1}\right) &:&=\sqrt{3}\quad \&\quad
% r_{2}\left( x_{2}\right) :=\left\{ 
% \begin{array}{c}
% 1,\quad \text{if}\ x_{2}\geq 0 \\ 
% \\ 
% 2,\quad \text{if}\ x_{2}\leq 0%
% \end{array}%
% \right. , \\
% ii)\qquad \qquad \qquad \beta &:&=\sqrt{3}\left\vert x_{1}\right\vert
% +r_{2}\left\vert x_{2}\right\vert .\qquad
% \end{eqnarray*}%
So that the SCP property is satisfied:
\[ \lim\limits_{(x_1, x_2) \rightarrow (0, 0)} \dfrac{a(x_1, x_2)}{\beta(x_1, x_2)} = \lim\limits_{(x_1, x_2) \rightarrow (0, 0)} \frac{x_1 \left(\sqrt{3} \dfrac{x_2}{1 + x_2^2}\right) + x_2 \left(\dfrac{x_1}{1 + x_1^2} + \dfrac{x_2^2}{1 + x_2^2}\right)}{\sqrt{3} \vert x_1 \vert + \vert x_2 \vert r_2} = 0,
\]
also the CLF property is satisfied:
\[ a(x_1, x_2) + \min\limits_{u \in H} \{x_1 u_1 + x_2 u_2 \} <0, \mbox{ for all } (x_1, x_2) \neq (0, 0).\]
The CLF and SCP properties allow the design of the continuous stabilizer $u^\varepsilon (x) := (u_1^\varepsilon(x), u_2^\varepsilon(x)) \in H$.

The optimal stabilizer (with the best rate) is $(\overline{\omega}_1, \overline{\omega}_2) \in H$, defined by the equality
\[
\begin{aligned}
    \min\limits_{u \in H} \dot{V} 
        &= a(x) + b_1 \overline{\omega}_1 + b_2 \overline{\omega}_2 \\
        &= a(x) - (\vert b_1 \vert r_1 + \vert b_2 \vert r_2),
\end{aligned}
\]
consequently (see example in \cite{Leyva2013})
\[
\overline{\omega}(b_1, b_2) = 
    \begin{cases}
    v_{1} & \mbox{ if } (b_1, b_2) \in \mathrm{cl}(\mathbb{R}^2_{++}), \\
    v_{2} & \mbox{ if } (b_1, b_2) \in \mathrm{cl}(\mathbb{R}^2_{-+}), \\
    v_{3} & \mbox{ if } (b_1, b_2) \in \mathrm{cl}(\mathbb{R}^2_{--}), \\
    v_{4} & \mbox{ if } (b_1, b_2) \in \mathrm{cl}(\mathbb{R}^2_{+-}). \\
    \end{cases}
\]
so that, by a straightforward calculation we get that,
\[ \min\limits_{u \in H} \dot{V} = a(x) - (\vert b_1 \vert r_1 + \vert b_2 \vert r_2) < 0, \mbox{ if } x\neq 0. \]

The CLF and SCP properties allow the design of the continuous stabilizer $u^\varepsilon (x) := (u_1^\varepsilon (x), u_2^\varepsilon (x)) \in H$.

\begin{figure}[thb]
\centering
    \includegraphics[width=0.8\textwidth]{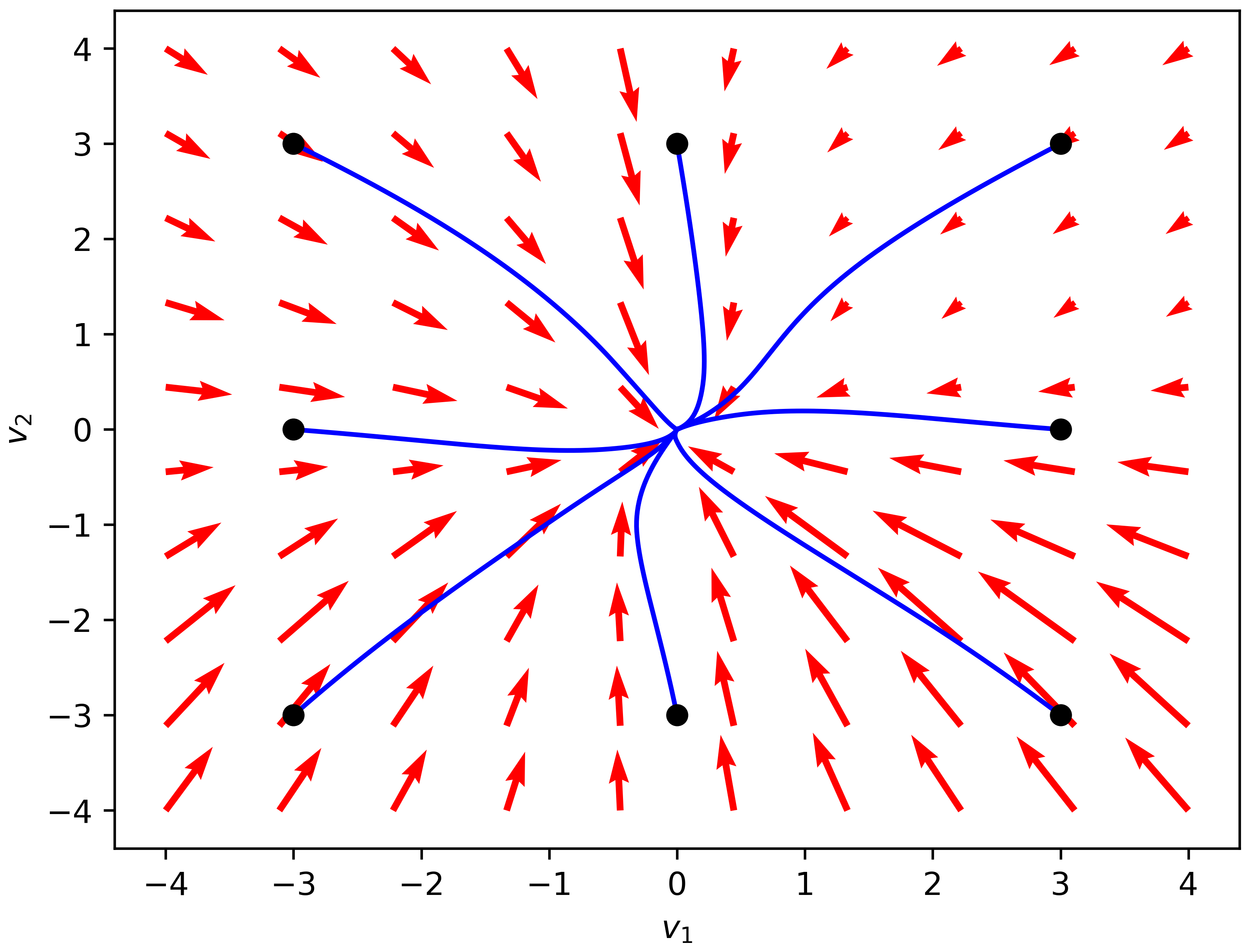}
    \caption{Phase portrait of the system, with feedback controls restricted to the triangle.}
    \label{Figure:phase-portrait}
\end{figure}

The script used to plot Figure \ref{Figure:phase-portrait} is available in \cite{Leyva-Aguirre-Espinoza:2022}.

\section{Results and discussion}

In the present work, we address the problem of the global stabilization of an affine system thought a continuous feedback function restricted to an $m$-dimensional CVS $U_{\phi }$ convex and bounded, represented as a sub-level set
\[ U_\phi := \left\{ u \in \mathbb{R}^{m} \mid \phi(u) \leq 1 \right\}, \]
such that $0 \in \mathrm{int}\, U_{\phi}$, and $\phi : \mathbb{R}^{m} \rightarrow \mathbb{R}$ is a non-negative positively homogeneous function.

We address specially the case of convex polytopes, defined through the inequality $\phi(u) := \max\limits_i \left\{ v_i^T u \right\} \leq 1$, where $\phi (u)$ is a continuous piecewise linear function. For this case in concrete, we show the solution to the CLF-optimization problem (\ref{equation:CLF}), represented by a feedback function $\omega (x)$ taking values at the vertices of the polytope, on such way that is not admissible because it is discontinuous.

In general, for any CVS $U_{\phi}$ we can design an admissible control for the system (\ref{equation:M_affine_system}) using an explicit formula of admissible feedback $u^{\varepsilon} (x)$ with CVS given by an $m$-dimensional asymmetric hyperbox $H$, designed to stabilize globally the affine system (\ref{equation:affine_system}) under the CLF and SCP conditions of the Artstein's theorem, in such way that restricted to $U_{\phi} \subset H$ we get an admissible feedback $w^{\varepsilon}(x)$ that stabilize globally (\ref{equation:M_affine_system}) for some value $M > 1$. Some properties of $u^{\varepsilon}(x) \in H$, such as continuity and the extreme values reaching, are inherited to $w^{\varepsilon}(x)$; in such way that if $u^{\varepsilon} (x) \in \partial H$, then  $w^{\varepsilon} (x) \in \partial U_{\phi}$.

%Given a stabilization problem composed of an affine system (\ref{equation:affine_system}) and any compact CVS $U_{\phi}$, with $0 \in \mathrm{int\,} U_{\phi} \subset \mathbb{R}^{m}$, we describe an approach that allows us to use the design presented in \cite[Theorem 14]{Leyva2013} and summarized as an algorithm in Section \ref{section:example}. In addition to the conditions established in \cite[Theorem 14]{Leyva2013}, we must seek to represent the CVS set as a level set of a non-negative function $\phi(u)$, so that
%\begin{equation*}
%    U_{\phi} = \{u \in \mathbb{R}^{m} \mid \phi(u) \leq 1\},
%\end{equation*}

%Once a hyperbox $H$ has been determined, it is possible to generate the admissible stabilizer $u^{\varepsilon}(x) \in H$ described in formulas (\ref{equation:feedback}) - (\ref{equation:tau}), to obtain by means of (\ref{equation:M_feedback}) the feedback function $w^{\varepsilon}(x)$
%restricted to CVS $U_{\phi} \subset H$.

%The sub-optimality, continuity and structurally decentralized properties of the control $u^{\varepsilon}(x)$ are inherited to the control $w^{\varepsilon}(x)$.

\bibliographystyle{siam}
\bibliography{bibliography}
\end{document}